\documentclass[preprint]{amsart}
\usepackage{amssymb}
\usepackage{amsthm}
\usepackage{amssymb}
\usepackage{mathtools}
\usepackage{xparse}
\usepackage{pdfsync}

\newtheorem{lemma}{Lemma}
\newtheorem{theorem}{Theorem}

\newtheorem{corollary}{Corollary}

\newtheorem*{lemma*}{Lemma}
\newtheorem*{theorem*}{Theorem*}
\newtheorem*{proposition*}{Proposition}
\newtheorem*{corollary*}{Corollary}
\newtheorem*{conjecture*}{Conjecture}

\theoremstyle{definition}
\newtheorem*{definition*}{Definition}

\newtheorem*{remark*}{Remark}
\newtheorem*{assumption*}{Assumption}

\author{P. Fortuny Ayuso}
\address{Dpt. of Mathematics, Univ. of Oviedo, Oviedo, Spain.}
\email{fortunypedro@uniovi.es}
\date{\today}
\title[Number of Puiseux exponents of an invariant curve]{On the number of Puiseux exponents of an invariant branch of a vector field}
\subjclass{32S05, 32S65, 14H20}
\begin{document}
%\begin{frontmatter}
%\author{P. Fortuny Ayuso}
%\address{Dpt. of Mathematics, Univ. of Oviedo, Oviedo, Spain.}
%\ead{fortunypedro@uniovi.es}

%\author{J. Rib\'on}
%\date{\today}
%\title{On the number of Puiseux exponents of an invariant branch of a vector field}
\maketitle
\begin{abstract}
  We show that the multiplicity of a plane analytic $1-$form is a bound for the number of Puiseux exponents of a (formal or convergent) branch. This is true whether the associated foliation is dicritical or not. %In the non-dicritical case we can improve the bound.
\end{abstract}

% \begin{keyword}
%   Singular Holomorphic Foliation\sep
%   Invariant curve\sep
%   Puiseux exponents\sep
%   Multiplicity\sep
%   Poincaré Problem\sep
%   \MSC[2010] 32S05\sep
%   \MSC[2010] 32S65\sep
%   \MSC[2010] 14H20\sep
% \end{keyword}
%\end{frontmatter}

\maketitle
\section{Introduction. The Newton-Puiseux polygon}
Among the problems related to the complexity of the invariant curves of a germ of singular analytic foliation in the plane ($1-$form or, equivalently, analytic vector field) ---the most famous one being the Poicaré Problem, see \cite{Poincare2},  \cite{Cerveau-Neto-1991} and \cite{Carnicer}, for example--- one of the open questions is whether the number of Puiseux exponents of such a curve can be bounded in terms of the local invariants of the singularity of the foliation. In this note we prove that the multiplicity of the singularity of the $1-$form is such a bound: an invariant branch can have at most as many Puiseux exponents as the minimum order of the coefficients of the $1-$form plus one.
%We can improve the bound if we assume the vector field is non-dicritical.

The main tool is the Newton-Puiseux polygon, whose construction can be found in \cite{Ince} and, more adapted to the modern notation, in \cite{CanoJ}. We give, in this introduction, the most concise summary of its construction, for the sake of completeness. Notice that we omit general arguments about existence and convergence because they are of no use to us.

\subsection{The Newton-Puiseux construction}
In the most general case we shall need, we consider a formal $1-$form
\begin{equation}
  \label{eq:1-form}
  \omega = a(x,y) dx + b(x,y) dy
\end{equation}
where $a(x,y)$ and $b(x,y)$ are power series in $y$ whose coefficients belong to some ring of formal power series $\mathbb{C}[[x^{1/n}]]$ for some $n\in \mathbb{N}$. We assume $a(0,0)=b(0,0)=0$ (i.e. the form is \emph{singular}). Let $\Gamma=\sum f_kx^{k/m}$ be a formal power series with $k\in \mathbb{N}$ for $k\geq m$ and $m\in \mathbb{N}$ (i.e. $\Gamma$ is a Puiseux expansion of a formal branch transverse to $x=0$). We say that $\Gamma$ is \emph{invariant} for $\omega$ if
\begin{equation*}
  \textstyle
  a\big(x, \sum f_k x^{k/m}\big) +
  b\big(x, \sum f_kx^{k/m}\big)\big(\sum k f_kx^{k/m-1}\big) = 0.
\end{equation*}
Given $\omega$, we construct its \emph{cloud of points} as the set
\begin{equation*}
  \mathcal{C}(\omega) = \left\{ (i,j)\in  \frac{\mathbb{Z}_{\geq -1}}{n}
    \times \mathbb{N}\,:\,
  a_{i,j}\neq 0 \mbox{ or } b_{i+1,j-1}\neq 0
  \right\}
\end{equation*}
where $a(x,y)=\sum a_{ij}x^iy^j$ and $b(x,y)=\sum b_{ij}x^iy^j$. The \emph{Newton Polygon} of $\omega$ is the following set:
\begin{equation*}
  \mathcal{N}(\omega) = \mathrm{convex\ envelope} \left( \left\{
  (i,j) + \mathbb{R}_{\geq 0}\times \mathbb{R}_{\geq 0} \,:\, (i,j) \in \mathcal{C}(\omega)
    \right\} \right).
\end{equation*}
For a rational number $\mu\in \mathbb{Q}$, with $\mu\geq 1$, let $L_{\mu}$ be the unique line of slope $-1/\mu$ (we say that $L_{\mu}$ has \emph{co-slope} $\mu$) in $\mathbb{R}^2$ which meets $\mathcal{N}(\omega)$ only at its topological border and let $(\tau,0)$ be the point at which $L_{\mu}$ meets the $OX$ axis. Let $\omega_{\mu}=a(x, cx^{\mu} + \overline{y})dx + b(x, cx^{\mu} + \overline{y}) d (cx^{ \mu}+\overline{y})$ be the $1-$form corresponding to the change of variables $y= cx^{\mu} + \overline{y}$. The following results are well-known \cite{CanoJ}:
\begin{lemma}
  If there is an invariant curve whose Puiseux expansion starts with $cx^{\mu}$ then the Newton polygon $\mathcal{N}(\omega_{\mu})$ of $\omega_{\mu}$ meets $OX$ only at points with abscissa strictly greater than $\tau$, if at all.
\end{lemma}
Considering the branch $\Gamma\equiv \sum_{k\geq m} f_kx^{k/m}$, we may define, recursively, $\omega_{m-1}=\dots=\omega_1=\omega_0=\omega$ and, for $k\geq m$:
\begin{equation*}
  w_k = w_{k-1}(x, f_k x^{k/m} + \overline{y})
\end{equation*}
and, by recurrence, we know that if $\Gamma$ is invariant for $\omega$ then for all $k$, the line $L_{k/m}=L_{k/m}(\omega_{k-1})$ meets $OX$ strictly to the left of $\mathcal{N}(\omega_{k})$ (the following Newton polygon).

Moreover, each time a coefficient $f_{k}$ gives rise to a side on the following polygon $\mathcal{N}(\omega_{k})$, the next coefficient $f_{k+1}$ comes from a point strictly lower than the previous one:
\begin{lemma}\label{lem:lower-side}
  If $L_{k/m}$ meets $\mathcal{N}(\omega_{k})$ on a side, then the highest point of $L_{(k+1)/m}\cap \mathcal{N}(\omega_{k})$ is strictly lower than the highest point of $L_{k/m}\cap \mathcal{N}(\omega_{k-1})$.
\end{lemma}

\section{The multiplicity bounds the number of Puiseux exponents}
Let $\omega$ and $\Gamma$ be as above (admitting rational exponents in $x$, but with a common denominator). We define the $y-$order of $\omega$ as the the ordinate of the highest point $(i,j)\in \mathcal{N}(\omega)\cap L_1$, where $L_1$ is the line of co-slope $1$ meeting $\mathcal{N}(\omega)$ on its border. Consider $\omega_k$ for $k\in \mathbb{N}$, as above. For each $k\in \mathbb{N}$, let $q_k$ be the product of the denominators of the Puiseux exponents up of $\Gamma$ to $k/m$. The \emph{multiplicity of $\omega$} is the smallest order of $a(x,y), b(x,y)$ plus one.
\begin{theorem}\label{the:bound-number-puiseux-exponents}
  If the branch $\Gamma$, transverse to $x=0$, is invariant for $\omega$ and has $r$ Puiseux exponents, then the $y-$order of $\omega$ is at least $r$. As a consequence, the multiplicity of $\omega$ is at least the largest number of Puiseux exponents of an invariant branch.
\end{theorem}
Before proceeding, we require two elementary results.
\begin{lemma}\label{lem:strict-decrease}
  Let $P=(i,j)$ be the highest point of the line $L_{k/m}$ meeting the Newton polygon of $\omega_{k-1}$ at its border. Assume $f_{k}\neq 0$ and that $q_{k}=sq_{k-1}$ with $s>1$. If $j>1$, then the Newton polygon of $\omega_{k}$ contains both $(i,j)$ and $(i+(j-1) \frac{k}{m}, 1)$ if $s\geq j$ or $(i+(s-1) \frac{k}{m}, j-(s-1))$ otherwise.
\end{lemma}
\begin{proof}
  Let $t=\max\{1,j-(s-1)\}$. As $s>1$, the points $(i+l \frac{k}{m}, j-l)$ do not belong to the cloud of points of $\omega_{k-1}$, for $l=1,\dots, j-t$. The fact that $P$ is the highest vertex of the polygon in $L_{k/m}$ implies that either $a_{ij}\neq 0$ or $b_{i+1,j-1}\neq 0$. In any case, the coordinate change $y=f_{k}x^{k/m}+\overline{y}$ gives rise, for the point $(i+(j-t) \frac{k}{m}, t)$ \emph{only} to the terms
  \begin{equation}\label{eq:point-one-below-1}
    f_{k}^{j-t}\left(\binom{j}{j-t} a_{ij} + \frac{k}{m}\binom{j-1}{j-t-1}
      b_{i+1,j-1}\right)x^{i+(j-t)k/m}\overline{y}^{t}dx
  \end{equation}
  and
  \begin{equation}\label{eq:point-one-below-2}
    f_{k}^{j-t}\binom{j-1}{j-t}b_{i+1,j-1}x^{i+(j-t)k/m+1}\overline{y}^{t-1}d\overline{y}
  \end{equation}
  (notice that $1\leq j-t\leq j-1$ and $j-t-2\geq 0$). In order for this point not to appear in the new Newton Polygon, both expressions must be $0$. We know that $f_{k}\neq 0$, so that necessarily, $b_{i+1,j-1}=0$, because \eqref{eq:point-one-below-2} must be $0$ and this implies that $a_{ij}=0$ in \eqref{eq:point-one-below-1}, which prevents $P$ from being in the Newton polygon of $\omega$, a contradiction.
\end{proof}
\begin{corollary}
  As a consequence, the highest point of $L_{(k+1)/m}$ is at height either $j-(s-1)$ or $1$.
\end{corollary}
\begin{proof}
  Because the segment joining $P=(i,j)$ and $Q=(i+t \frac{k}{m},j-t)$ in the previous proof has co-slope $k/m$, the only way to continue following $\Gamma$ as a solution of $\omega$, by Lemma \ref{lem:lower-side} is either using a vertex which is $Q$ or lower, or a side which starts at $Q$ or lower, which gives the result.
\end{proof}
\begin{proof}[Proof of the theorem]
  The theorem is now proved by recurrence. As $\Gamma$ is not tangent to the $OY$ axis, its Puiseux expansion $\Gamma\equiv \sum f_kx_k^{k/m}$ starts with $k\geq m$. Let $j$ be the $y-$order of $\omega=\omega_0=\dots=\omega_{m-1}$. By Lemma \ref{lem:strict-decrease}, the first Puiseux exponent $\mu_1 = k_1/m$ gives rise, in the next Newton polygon (that of $\omega_{k_1}$) to a side  of co-slope $\mu_1$ whose lowest vertex has height strictly less than $j$ unless $j=1$. By recurrence, one sees that, if $r\geq j$, then the $j-1-$th Puiseux exponent gives rise (on the appropriate Newton polygon) to a side whose highest vertex is at height $1$. At this point, it is well known \cite{CanoJ} that only one more Puiseux exponent can appear, and we are done.

  As the $y-$order is less than or equal to the order of an analytic differential form, the consequence follows easily.
\end{proof}

\end{document}